\documentclass[a4paper,leqno]{article}

\usepackage[T1]{fontenc}
\usepackage[utf8]{inputenc}
\usepackage{mathtools,amssymb,amsthm}
\usepackage{lmodern}
\usepackage{hyperref}
\usepackage{microtype}
\usepackage{color}
\usepackage{enumerate}

\theoremstyle{plain}
\newtheorem{thm}{Theorem}[section]
\newtheorem{lem}[thm]{Lemma}
\newtheorem{coro}[thm]{Corollary}
\newtheorem{prop}[thm]{Proposition}

\newtheorem{thmalph}{Theorem}

\theoremstyle{definition}
\newtheorem{rmrk}[thm]{Remark}
\theoremstyle{remark}

\newcommand{\R}{\mathbb{R}}
\newcommand{\Z}{\mathbb{Z}}
\newcommand{\N}{\mathbb{N}}
\newcommand{\F}{\mathbb{F}}
\renewcommand{\H}{\mathbb{H}}
\newcommand{\C}{\mathbb{C}}
\renewcommand{\O}{\mathbb{O}}
\newcommand{\calD}{\mathcal{D}}
\newcommand{\calE}{\mathcal{E}}
\newcommand{\calF}{\mathcal{F}}
\newcommand{\calO}{\mathcal{O}}
\newcommand{\calP}{\mathcal{P}}
\newcommand{\calY}{\mathcal{Y}}

\newcommand*{\rom}[1]{\romannumeral}

\providecommand{\msc}[1]{{\noindent\small\textbf{Mathematics Subject Classification (2020)} --- #1.}}
\providecommand{\keywords}[1]{{\noindent\small\textbf{Keywords} --- #1.}}

\DeclareMathOperator{\upO}{O}

\DeclareMathOperator{\SU}{SU}
\DeclareMathOperator{\U}{U}
\DeclareMathOperator{\Spin}{Spin}
\DeclareMathOperator{\Sp}{Sp}
\DeclareMathOperator{\IM}{Im}
\DeclareMathOperator{\RE}{Re}

\DeclareMathOperator{\Id}{Id}

\DeclareMathOperator{\1}{\mathbf{1}}
\DeclareMathOperator{\0}{\mathbf{0}}

\DeclareMathOperator*{\Res}{Res}

\DeclareMathOperator{\supp}{supp}

\DeclareMathOperator{\Ind}{Ind}

\DeclareMathOperator{\const}{const}

\numberwithin{equation}{section}

\providecommand{\contact}{{
\bigskip
\small
\noindent
\textsc{J.~Frahm}:
Department of Mathematics, Aarhus University, Ny Munkegade 118,  8000, Aarhus C, Denmark
\par\noindent\nopagebreak
\textit{E-mail address:} \href{mailto:frahm@math.au.dk}{\texttt{frahm@math.au.dk}}\\

\noindent
\textsc{P.~Spilioti}:
Mathematisches Institut, Georg--August Universit\"{a}t G\"{o}ttingen, Bunsenstra{\ss}e 3--5, 37073 G\"{o}ttingen, Germany
\par\noindent\nopagebreak
\textit{E-mail address:} \href{mailto:polyxeni.spilioti@mathematik.uni-goettingen.de}{\texttt{polyxeni.spilioti@mathematik.uni-goettingen.de}}
}}

\title{Resonances and residue operators for pseudo-Riemannian hyperbolic spaces}
\author{Jan Frahm, Polyxeni Spilioti}

\begin{document}

\maketitle

\begin{abstract}
For any pseudo-Riemannian hyperbolic space $X$ over $\R,\C,\H$ or $\O$, we show that the resolvent $R(z)=(\Box-z\Id)^{-1}$ of the Laplace--Beltrami operator $-\Box$ on $X$ can be extended meromorphically across the spectrum of $\Box$ as a family of operators $C_c^\infty(X)\to \calD'(X)$. Its poles are called \emph{resonances} and we determine them explicitly in all cases. For each resonance, the image of the corresponding residue operator in $\calD'(X)$ forms a representation of the isometry group of $X$, which we identify with a subrepresentation of a degenerate principal series. Our study includes in particular the case of even functions on de Sitter and Anti-de Sitter spaces.

For Riemannian symmetric spaces analogous results were obtained by Miatello--Will and Hilgert--Pasquale. The main qualitative differences between the Riemannian and the non-Riemannian setting are that for non-Riemannian spaces the resolvent can have poles of order two, it can have a pole at the branching point of the covering to which $R(z)$ extends, and the residue representations can be infinite-dimensional.
\end{abstract}

\renewcommand{\abstractname}{Résumé}

\begin{abstract}
Pour tout espace hyperbolique pseudo-Riemannien $X$ sur $\R,\C,\H$ ou $\O$, cet article montre que la résolvante $R(z)=(\Box-z\Id)^{-1}$  de l'opérateur de Laplace--Beltrami $-\Box$ sur $X$ peut être étendue méromorphiquement à travers le spectre de $\Box$ en une famille d'opérateurs $C_c^\infty(X)\to \mathcal{D}'(X)$. Ses pôles sont appelés \emph{résonances} et nous les déterminons explicitement dans chacun des cas. Pour chaque résonance, l'image de l'opérateur du résidu correspondant dans $\mathcal{D}'(X)$ forme une représentation du groupe d'isométrie de $X$, que l'on identifie avec une sous-représentation d’une série principale dégénérée. Cette étude inclus le cas particulier des fonctions paires sur les espaces de de Sitter et Anti-de Sitter.

Pour des espaces symétriques Riemannien des résultats analogues ont été obtenus par Miatello--Will et Hilgert--Pasquale. Les principales différences entre le cas Riemannien et pseudo-Riemannien sont que, dans le cas non-Riemannien, la résolvante peut avoir des pôles d'ordre deux, elle peut avoir un pôle aux points auxquels le revêtement où l’on étend $R(z)$ ramifie, et les représentations du résidu peuvent être de dimension infinie.  
\end{abstract}

\bigskip
\newline
\msc{Primary: 43A85; Secondary: 22E46, 58J50}


\section*{Introduction}

Resonances of the Laplace--Beltrami operator $-\Delta$ on a Riemannian manifold $X$ are the poles of a possible meromorphic extension of its resolvent
$$ R(z)=(\Delta-z\Id)^{-1} \qquad (z\in\C\setminus\sigma(\Delta)) $$
across the spectrum $\sigma(\Delta)\subseteq\C$ of $\Delta$. Their position in the complex plane (or a certain Riemann surface) carries refined information about the long-term behavior of the corresponding dynamical system, so they are important spectral invariants of the manifold $X$. Consequently, there has been a lot of work on resonances for special classes of non-compact Riemannian manifolds, for instance for asymptotically hyperbolic manifolds by Mazzeo--Melrose~\cite{MM87}, Guillop\'{e}--Zworski~\cite{GZ95} and Guillarmou~\cite{Gui05}, and for Riemannian symmetric spaces by Miatello--Wallach~\cite{MW92}, Mazzeo--Vasy~\cite{MV05}, Strohmeier~\cite{Str05}, Hilgert--Pasquale~\cite{hilgert2009resonances} and Hilgert--Pasquale--Przebinda~\cite{HPP16,HPP17a,HPP17b}.

Pseudo-Riemannian manifolds also come with a Laplace--Beltrami operator $-\Box$. Although $\Box$ no longer elliptic and positive, it still makes sense to consider its resolvent and study its meromorphic continuation. In the pseudo-Riemannian setting, however, much less seems to be known (see e.g. \cite{BMB93,SBZ97} for some Lorentzian manifolds modelling black holes).\\

In this paper, we study resonances of the Laplace--Beltrami operator on a certain family of pseudo-Riemannian manifolds. More precisely, let $X$ be a hyperbolic space, i.e. a reductive symmetric space of the form
$$ X=G/H=\U(p,q;\F)/\big(\U(1;\F)\times\U(p-1,q;\F)\big) \qquad (\F=\R,\C,\H,\O). $$
(Note that the case $\F=\O$ only makes sense for $p+q\leq3$, see Section~\ref{sec:HyperbolicSpaces} for details.) Then $X$ carries a unique (up to scalar multiples) $G$-invariant pseudo-Riemannian metric of signature $(dq,d(p-1))$, where $d=\dim_\R\F\in\{1,2,4,8\}$. The corresponding Laplace--Beltrami operator $-\Box$ is self-adjoint on $L^2(X)$ (with respect to the Riemannian measure) and we consider its resolvent
$$ R(z)=(\Box-z\Id)^{-1} \qquad (z\in\C\setminus\sigma(\Box)). $$

For $q=0$, the space $X$ is a compact Riemannian manifold, so $\sigma(\Box)$ is discrete and consists of eigenvalues, hence the resolvent $R(z)$ is already meromorphic on $\C$. If $p=1$, then $X$ is a non-compact Riemannian symmetric space and $\Box$ has purely continuous spectrum. In this case, the meromorphic extension of $R(z)$ to a Riemann surface was studied in detail by Miatello--Will~\cite{MW00}, Mazzeo--Vasy~\cite{MV05} and Hilgert--Pasquale~\cite{hilgert2009resonances}.

We therefore assume $p\geq2$ and $q\geq1$, so that $X$ is non-compact and non-Riemannian, and put $\rho=\frac{d(p+q)-2}{2}$. In this case, the spectrum $\sigma(\Box)$ of $\Box$ is of the form $\sigma_c(\Box)\cup\sigma_p(\Box)$, with $\sigma_c(\Box)=[\rho^2,\infty)$ the continuous spectrum and
$$ \sigma_p(\Box) = \begin{cases}\{\rho^2-s^2:s\in(\rho+2\Z+1)\cap\R_+\}&\mbox{for $\F=\R$ and $q$ odd,}\\\{\rho^2-s^2:s\in(\rho+2\Z)\cap\R_+\}&\mbox{for $\F=\R$ and $q$ even, or $\F=\C,\H,\O$,}\\\end{cases} $$
the discrete set of eigenvalues (see Theorem~\ref{thm:InversionFormula}). We show that $R(z)$ has a meromorphic extension to a double cover of $\C\setminus\{\rho^2\}$, cut across the half line $[\rho^2,\infty)$. To this end we define
$$ \widetilde{R}(\zeta) = R(\rho^2+\zeta^2) = (\Box-\rho^2-\zeta^2)^{-1} \qquad (\IM\zeta>0,\zeta^2+\rho^2\not\in\sigma_p(\Box)), $$
then the extension of $R(z)$ across $\sigma(\Box)$ to a double cover of $\C\setminus\{\rho^2\}$ is equivalent to the extension of $\widetilde{R}(\zeta)$ across the real line.

\begin{thmalph}[see Theorem~\ref{thm:MeroCont}]
	The resolvent $\widetilde{R}(\zeta):C_c^\infty(X)\to\calD'(X)$, initially defined for $\IM\zeta>0$ with $\zeta^2+\rho^2\not\in\sigma_p(\Box)$, has a meromorphic extension to $\zeta\in\C$.
\end{thmalph}

To describe the resonances of $\Box$, i.e. the poles of the meromorphic extension of $\widetilde{R}(\zeta)$, we divide the complex plane into three regions: the upper half plane $\{\IM\zeta>0\}$, the real line $\{\IM\zeta=0\}$ and the lower half plane $\{\IM\zeta<0\}$.

\begin{thmalph}[see Theorems~\ref{thm:ResonancesUpperHalfPlane}, \ref{thm:ResonancesRealLine}, \ref{thm:ResonancesLowerHalfPlane1}, \ref{thm:ResonancesLowerHalfPlane2}, \ref{thm:ResonancesLowerHalfPlane3} and \ref{thm:ResonancesLowerHalfPlane4}]
	The poles of $\widetilde{R}(\zeta)$ are:
	\begin{enumerate}
		\item In the upper half plane $\{\IM\zeta>0\}$ there are single poles at $\zeta=is$, $\rho^2-s^2\in\sigma_p(\Box)$.
		\item On the real line $\{\IM\zeta=0\}$ the only possible pole is $\zeta=0$, and this is indeed a single pole if and only if either $\F=\R$ and $p-q\in4\Z+2$ or $\F=\C$ and $p-q\in2\Z$.
		\item In the lower half plane $\{\IM\zeta<0\}$ the poles are:
		\begin{enumerate}
			\item For $\F=\R$ with $p$ and $q$ even or $\F=\C,\H,\O$ there are single poles at $-is$, $s\in\rho+2\N$.
			\item For $\F=\R$ with $p$ odd and $q$ even there are single poles at $-is$, where either $s\in\rho+2\N$ or $s\in(\rho+2\Z+1)\cap\R_+$.
			\item For $\F=\R$ with $p$ and $q$ odd there are no poles in the lower half plane.
			\item For $\F=\R$ with $p$ even and $q$ odd there are single poles at $-is$, $0<s<\rho$, $s\in\rho+2\Z$, and there are poles of order two at $-is$, $s\in\rho+2\N$.
		\end{enumerate}
	\end{enumerate}
\end{thmalph}

In particular, all singularities of $\widetilde{R}(\zeta)$ are on the imaginary axis $\zeta=is$, $s\in\R$. The residue of $\widetilde{R}(\zeta)$ at $\zeta=is$ is a convolution operator of the form
$$ C_c^\infty(X)\to\calD'(X), \quad f\mapsto f*\varphi $$
for a so-called \emph{$H$-spherical distribution $\varphi\in\calD'(X)$}, i.e. $\varphi$ is an $H$-invariant eigendistribution: $\Box\varphi=-(s^2-\rho^2)\varphi$. Since $\Box$ is $G$-invariant, the resolvent $\widetilde{R}(\zeta)$ and its residues are $G$-equivariant, so their images form subrepresentations of $\calD'(X)$ called \emph{residue representations}. We identify the residue representations in all cases. The statement of the results requires a certain amount of notation, and we refer to Section~\ref{sec:ResonancesAndRepresentations} for details. At this point we remark that for $\F=\R$ with $pq$ even or $\F=\C,\H,\O$ the residue representation at $\zeta=-is$, $s=\rho+2k$, $k\in\N$, is finite-dimensional, while all other residue representations are infinite-dimensional. This is one of the main differences to the Riemannian situation, i.e. the case $p=1$, where all residue representations are finite-dimensional (see \cite{hilgert2009resonances}). Moreover, in the Riemannian case all poles are of order one while we also encounter poles of order two in the case $\F=\R$ with $p$ even and $q$ odd.\\

Our method of proof is very similar to the one employed by Hilgert--Pasquale~\cite{hilgert2009resonances}. We use the explicit spectral decomposition of $\Box$ on $L^2(X)$ which is equivalent to the decomposition of the left-regular representation of $G$ on $L^2(X)$ into irreducible unitary representations, also referred to as Plancherel formula. This formula is due to Faraut~\cite{faraut1979distributions} for the cases $\F=\R,\C,\H$ and due to Kosters~\cite{Kos83} for the exceptional case $\F=\O$. The key point is the explicit description of the Plancherel measure for the continuous part which has to be extended meromorphically using a shift of contour of integration. In this way, we pick up residues coming from poles of the Plancherel density. In contrast to the Riemannian case, some of these residues actually cancel with the residues arising from the discrete spectrum of $\Box$, leading for instance to the fact that for $\F=\R$ with $p$ and $q$ odd the resolvent $\widetilde{R}(\zeta)$ is holomorphic in the lower half plane.

Let us also mention some related work for the case of Riemannian symmetric spaces. In addition to the work of Hilgert--Pasquale~\cite{hilgert2009resonances} who treated all Riemannian symmetric spaces of rank one, there have been some attempts at Riemannian symmetric spaces of higher rank by Hilgert--Pasquale--Przebinda~\cite{HPP16,HPP17a,HPP17b}. They were able to treat most cases of rank two. Working in a different direction, Will~\cite{Wil03} and Roby~\cite{Rob22} studied resonances of the Laplacian acting on sections of vector bundles over Riemannian symmetric spaces of rank one.\\

One might wonder about an extension of our results to homogeneous vector bundles over the pseudo-Riemannian symmetric space $X$. In fact, the only homogeneous vector bundles which carry an invariant Hermitian metric are the ones associated with an irreducible representation of the compact factor $\U(1;\F)$ of $H$. The corresponding Plancherel formula for such vector bundles was obtained by Shimeno~\cite{Shi95}. We expect that using his work, similar methods as used in this paper are applicable. This might be particularly interesting for the case $\F=\R$ where $\U(1;\F)=\upO(1)=\{\pm1\}$, so functions on $X=\upO(p,q)/\upO(1)\times\upO(p-1,q)$ can be identified with \emph{even} functions on the hyperboloid $\upO(p,q)/\upO(p-1,q)$. To also treat \emph{odd} functions, one has to consider the line bundle associated with the non-trivial character of $\upO(1)$. For $q=1$ this would describe resonances on de Sitter space $\operatorname{dS}^n=\upO(n,1)/\upO(n-1,1)$ and for $p=2$ this would treat Anti-de Sitter space $\operatorname{AdS}^n=\upO(2,n-1)/\upO(1,n-1)$. We plan to study these cases in a subsequent paper.

\paragraph{Acknowledgements.} Both authors were partly supported by a research grant from the Villum Foundation (Grant No. 00025373), and, in addition to that, the second author was partly supported by the DFG Research Training Group 2491 ``Fourier Analysis and Spectral Theory''.

\section{Pseudo-Riemannian hyperbolic spaces}

We introduce pseudo-Riemannian hyperbolic spaces over $\F=\R,\C,\H,\O$, their isometry groups, polar coordinates and the Poisson kernel. For more details, we refer the reader to the work of Faraut~\cite{faraut1979distributions} and Shimeno~\cite{Shi95} for $\F=\R,\C,\H$ and of Kosters \cite{Kos83} for $\F=\O$. 

\subsection{Hyperbolic spaces and isometry groups}\label{sec:HyperbolicSpaces}

Let $\F=\R,\C,\H$ and write $d=\dim_\R\F=1,2,4$. We further let $p\geq1$ and $q\geq0$, write $n=p+q$ and consider on $\F^n$ the standard sesquilinear form of signature $(dp,dq)$ given by
$$ [x,y]=\overline{y_1}x_1+\cdots+\overline{y_p}x_p-\overline{y_{p+1}}x_{p+1}-\cdots-\overline{y_{p+q}}x_{p+q} \qquad  (x,y\in\F^n). $$
On the open subset $\{y\in\F^n:[y,y]>0\}$, the pseudo-Riemannian metric
$$ ds^2 = -\frac{[dy,dy]}{[y,y]} $$
is invariant under dilations $y\mapsto y\lambda$, $\lambda\in\F^\times=\F\setminus\{0\}$, and hence it induces a pseudo-Riemannian metric $g$ of signature $(dq,d(p-1))$ on the corresponding open subset
$$ X = \{[y]\in P^{n-1}(\F):[y,y]>0\}$$
in the projective space $P^{n-1}(\F)=(\F^{n}\setminus\{0\})/\F^\times$. This constructs a family of pseudo-Riemannian manifolds $(X,g)$ called \emph{hyperbolic spaces}. Note that $X$ is non-compact if and only if $q\geq1$. Moreover, $X$ is Riemannian if and only if either $p=1$ or $q=0$.

The group $G=\U(p,q;\mathbb{F})$ (i.e. $\upO(p,q)$ for $\F=\R$, $\U(p,q)$ for $\F=\C$ and $\Sp(p,q)$ for $\F=\H$) preserves the form $[\cdot,\cdot]$ and hence it acts on $X$. In fact, $G$ acts transitively on $X$, so we can identify $X\simeq G/H$, where $H=\U(1;\mathbb{F})\times\U(p-1,q;\mathbb{F})$ is the stabilizer of the base point $x_0=[(1,0,\ldots,0)]$. In this way, $X$ becomes a semisimple symmetric space (although the group $G$ itself is not semisimple for $\F=\C$, but can be replaced by the semisimple $\SU(p,q)$).

The symmetric spaces constructed in this way are all \emph{isotropic}, i.e. the isotropy group $H$ at $x_0$ acts transitively on each level set in $T_{x_0}X$ of the quadratic form induced by the metric $g$. Actually, the above construction gives all but three isotropic symmetric spaces (see \cite[page 382]{Wol67}). The missing ones can be constructed in a similar fashion using the octonions $\F=\O$. More precisely, we let $G=\U(p,q;\O)$ for $(p,q)\in\{(3,0),(2,1),(1,2)\}$, meaning $\U(3;\O)=F_4$, the compact simple Lie group of type $F_4$, or $\U(2,1;\O)=\U(1,2;\O)=F_{4(-20)}$, the non-compact simple Lie group of type $F_4$ and real rank one. Further, let $H=\U(p-1,q;\O)$, using the interpretation $\U(2;\O)=\Spin(9)$ and $\U(1,1;\F)=\Spin_0(1,8)$. Then $X=G/H$ is an isotropic symmetric space called \emph{octonionic hyperbolic space} (see \cite[Chapter 3.4]{Kos83} for details). For $p=3$, it is compact and Riemannian, for $p=1$ it is non-compact and Riemannian, and for $p=2$ it is non-compact and pseudo-Riemannian.

In this paper, we will only be concerned with non-compact pseudo-Riemannian hyperbolic spaces, so from now on we assume that $p\geq2$ and $q\geq1$.

\subsection{Polar coordinates}

Realizing $G$ as $n\times n$ matrices with entries in $\F$, we choose the maximal compact subgroup $K=\U(p;\F)\times\U(q;\F)$ of $G$ (resp. $K=\Spin(9)$ for $G=F_{4(-20)}$) and consider the one-parameter group $A=\{a_t:t\in\R\}$ given by
$$ a_t = \begin{pmatrix}\cosh t&&\sinh t\\&\0_{n-2}&\\\sinh t&&\cosh t\end{pmatrix} \qquad (t\in\R). $$
The map $K\times A\to X,\,(k,a)\mapsto ka\cdot x_0$ is surjective and induces a diffeomorphism onto an open dense subset of $X$ (see \cite[page 403]{faraut1979distributions} and \cite[Lemma 3.9.1]{Kos83}):
$$ K/M_0\times(0,\infty)\to X, \quad (kM_0,t)\mapsto ka_t\cdot x_0, $$
where
$$ M_0 = \left\{\begin{pmatrix}a&&&\\&B&&\\&&C&\\&&&a\end{pmatrix}:a\in\U(1,\F),B\in\U(p-1;\F),C\in\U(q-1;\F)\right\}\subseteq K $$
is the centralizer of $A$ in $K$. Note that $K/M_0\simeq(S^{dp-1}\times S^{dq-1})/\U(1;\F)$, where $d=\dim_\R\F\in\{1,2,4,8\}$. In these coordinates, the Riemannian measure $dx$ on $X$ takes the form (see \cite[page 403]{faraut1979distributions} and \cite[Lemma 3.9.1]{Kos83})
$$ dx = \const\times\,A(t)\,dt\,db, $$
where $A(t)=(\cosh t)^{dp-1}(\sinh t)^{dq-1}$, $db$ is the normalized $K$-invariant measure on $K/M_0$ and the constant only depends on the normalization of the measures.

\subsection{The isotropic cone and the Poisson kernel}

The \emph{isotropic cone} $\{y\in\F^n\setminus\{0\}:[y,y]=0\}$ is invariant under multiplication by $\lambda\in\U(1;\F)$ from the right and we denote by $\Xi$ the corresponding quotient:
$$ \Xi = \{\eta\in\F^n\setminus\{0\}:[\eta,\eta]=0\}/\U(1;\F). $$
$G$ acts transitively on $\Xi$ and the stabilizer of $\xi_0=[(1,0,\ldots,0,1)]$ is equal to $MN$, where
\begin{align*}
	M &= \left\{\begin{pmatrix}a&&\\&B&\\&&a\end{pmatrix}:a\in\U(1,\F),B\in\U(p-1,q-1;\F)\right\},\\
	N &= \exp\left\{\begin{pmatrix}w&z^*&-w\\z&\0_{n-2}&-z\\w&z^*&-w\end{pmatrix}:z\in\F^{n-2},w\in\IM\F\right\},
\end{align*}
with $z^*=-\overline{z}^\top\1_{p-1,q-1}$ and $\IM\F=\{w\in\F:\RE(w)=0\}$ (see \cite[page 391]{faraut1979distributions} and \cite[Chapter 3.5]{Kos83}). We therefore identify $\Xi\simeq G/MN$.

We introduce the \emph{Poisson kernel} $P:X\times\Xi\to\C$ by
$$ P(x,\xi) = \frac{\big|[y,\eta]\big|}{\sqrt{[y,y]}}, $$
where $x=[y]\in X$ and $\xi=[\eta]\in\Xi$. For fixed $\xi\in\Xi$, its complex powers (as functions of $x\in X$ where $P(x,\xi)\neq0$) are eigenfunctions of the Laplace--Beltrami operator $-\Box$ on $X$ (see \cite[page 394]{faraut1979distributions} and note that $\Box$ in our notation corresponds to $-\Box$ in Faraut's notation):
$$ \Box P(x,\xi)^{s-\rho} = -(s^2-\rho^2)P(x,\xi)^{s-\rho}, \qquad (s\in\C), $$
where $\rho=\frac{d(p+q)}{2}-1$.


\section{Harmonic analysis on $X$}

We recall from \cite{faraut1979distributions,Shi95} the Fourier transform, the Poisson transform, the corresponding spherical distributions and the inversion formula for the hyperbolic spaces $X$. For this whole section we assume $p\geq2$ and $q\geq1$.

\subsection{Degenerate principal series representations}

We briefly discuss the degenerate principal series representations that occur in the Plancherel formula for $X$. Recall the subgroups $M$, $A$ and $N$ of $G$. The product $P=MAN$ is a maximal parabolic subgroup of $G$. For $s\in\C$ we define a character of $P=MAN$ by
$$ \chi_s(ma_tn) = e^{st} \qquad (m\in M,t\in\R,n\in N) $$
and consider the induced representation $\pi_s=\Ind_P^G(\chi_{-s})$, realized as the left-regular representation on the space
$$ \calE_s(\Xi) = \{f\in C^\infty(\Xi):f(\lambda\xi)=\lambda^{s-\rho}f(\xi)\mbox{ for all }\xi\in\Xi,\lambda>0\}. $$

Let
$$ B=\{[y]\in\Xi:|y_1|^2+\cdots+|y_p|^2=|y_{p+1}|^2+\cdots+|y_{p+q}|^2=1\}, $$
then $\Xi=\R_+B$ and therefore every function $f\in\calE_s(\Xi)$ is uniquely determined by its restriction to $B$. Note that $K$ acts transitively on $B$ and the stabilizer of $b_0=[(1,0,\ldots,0,1)]$ equals $M_0$, so we can identify $B\simeq K/M_0$. In particular, there is a unique normalized $K$-invariant measure $db$ on $B$. The bilinear form
$$ \calE_s(\Xi)\times\calE_{-s}(\Xi)\to\C, \quad (f_1,f_2)\mapsto\langle f_1,f_2\rangle=\int_B f_1(b)f_2(b)\,db $$
is invariant under $\pi_s\times\pi_{-s}$ (see \cite[Proposition 5.1]{faraut1979distributions} and \cite[Chapter 3.6]{Kos83}).

Under the action of $K$, the space $L^2(B)$ decomposes into a direct sum of invariant subspaces
$$ L^2(B) \simeq \widehat{\bigoplus_{\ell,m\in\N}}\calY_{\ell,m}, $$
where
$$ \calY_{\ell,m} = \{f\in C^\infty(B):\Delta_1f=-\ell(\ell+dp-2)f,\Delta_2f=-m(m+dq-2)f\} $$
with $\Delta_1$ resp. $\Delta_2$ denoting the Laplacian on the unit sphere $S(\F^p)$ resp. $S(\F^q)$ that factors to $B\simeq(S(\F^p)\times S(\F^q))/U(1;\F)$. We note that $\calY_{\ell,m}$ is in general not irreducible under the action of $K$, but it turns out that every irreducible subquotient of $\calE_s(\Xi)$ decomposes into a direct sum of $\calY_{\ell,m}$'s. The set $\Lambda=\{(\ell,m)\in\N^2:\calY_{\ell,m}\neq\{0\}\}$ depends on $\F$ as well as $p$ and $q$ (see \cite[page 399]{faraut1979distributions} and \cite[Chapter 3.9]{Kos83}):
$$ \Lambda = \begin{cases}\{(\ell,m)\in\N^2:\ell\equiv m\mod2,m\in\{0,1\}\}&\mbox{if $q=1$ and $\F=\R$,}\\\{(\ell,m)\in\N^2:\ell\equiv m\mod2,m\leq\ell\}&\mbox{if $q=1$ and $\F=\C,\H,\O$,}\\\{(\ell,m)\in\N^2:\ell\equiv m\mod2\}&\mbox{if $q>1$.}\end{cases} $$

%

\subsection{The Fourier transform}\label{sec:FourierTransform}

For $f\in C_c^\infty(X)$ the following integral converges whenever $\RE(s)>\rho-d$:
$$ \calF_sf(\xi) = \frac{1}{\Gamma(\frac{s-\rho+d}{2})}\int_X P(x,\xi)^{s-\rho}f(x)\,dx. $$
By \cite[Proposition 7.1]{faraut1979distributions} and \cite[Chapter 3.9]{Kos83}, the function $(s,\xi)\mapsto\calF_sf(\xi)$ is entire in $s\in\C$ and smooth in $\xi$, so $\calF_sf\in\calE_s(\Xi)$. Moreover, $\calF_s(\Box f)=-(s^2-\rho^2)\calF_sf$. The function $\calF_sf$ is called \emph{Fourier transform} of $f$.

Every $K$-finite function $f\in C_c^\infty(X)$ can be written as a linear combination of functions of the form $f(ka_t\cdot x_0)=Y(k\cdot b_0)F(t)$ with $Y\in\calY_{\ell,m}$ and $F\in C_c^\infty(\R)$ even or odd (depending on whether $m$ is even or odd, see \cite[page 403]{faraut1979distributions}). For such functions we have
\begin{equation}
	\calF_sf(b) = \beta_{\ell,m}(s)Y(b)\int_0^\infty\Psi_{\ell,m}(t,s)F(t)A(t)\,dt,\label{eq:FourierOnKfinite}
\end{equation}
where
\begin{align*}
	\Psi_{\ell,m}(t,s) 
	={}& \sinh^m t\cosh^\ell t\\
	&\times{_2F_1}\left(\frac{\rho+s+m+\ell}{2},\frac{\rho-s+m+\ell}{2};m+\frac{dq}{2};-\sinh^2t\right)
\end{align*}
with ${_2F_1}(a,b;c;z)$ denoting the hypergeometric function and
$$ \beta_{\ell,m}(s) = \const\times(-1)^m\frac{(\frac{s-\rho}{2})(\frac{s-\rho}{2}-1)\cdots(\frac{s-\rho}{2}-\frac{m+\ell}{2}+1)}{\Gamma(m+\frac{dq}{2})\Gamma(\frac{s-\rho+\ell-m+dp}{2})}, $$
the constant being positive and  independent of $\ell$ and $m$ (see \cite[page 405]{faraut1979distributions}, \cite[Proposition 5.1]{Shi95} and \cite[Chapter 3.9]{Kos83}).


\subsection{The Poisson transform}\label{sec:PoissonTransform}

We further define the \emph{Poisson transform} $\calP_sg$ of a function $g\in\calE_s(\Xi)$:
$$ \calP_sg(x) = \frac{1}{\Gamma(\frac{-s-\rho+d}{2})}\int_BP(x,b)^{-s-\rho}g(b)\,db. $$
Then $\calP_sg\in C^\infty(X)$ is an eigenfunction of $\Box$:
$$ \Box(\calP_sg) = -(s^2-\rho^2)\calP_sg $$
and $\calP_sg$ depends holomorphically on $s\in\C$ (see \cite[Remark 4.5]{Shi95}). For $g=Y\in\calY_{\ell,m}$ we have (see \cite[Lemma 4.3]{Shi95}):
\begin{equation}
	\calP_sg(ka_t\cdot x_0) = \const\times\beta_{\ell,m}(-s)Y(k\cdot b_0)\Psi_{\ell,m}(t,s).\label{eq:PoissonOnKfinite}
\end{equation}


\subsection{Spherical distributions}

Following \cite[page 403]{faraut1979distributions}, we define a distribution $\varphi_s$ on $X$ by
\begin{equation}
	\langle\varphi_s,f\rangle = \calP_s(\calF_sf)(x_0) \qquad (f\in C_c^\infty(X)).\label{eq:DefZeta}
\end{equation}
Then
$$ \calP_s(\calF_sf) = f*\varphi_s, $$
the convolution of $f$ with $\varphi_s$. The distributions $\varphi_s$ are \emph{spherical}, i.e. they are $H$-invariant and eigendistributions of the Laplacian (see \cite[Proposition  5.4]{faraut1979distributions} and \cite[Proposition 3.7.2]{Kos83}).
$$ \Box\varphi_s=-(s^2-\rho^2)\varphi_s. $$
They have the following properties:

\begin{lem}[{see \cite[Th\'{e}or\`{e}me 7.3]{faraut1979distributions} and \cite[Proposition 3.9.4]{Kos83}}]\label{lem:SphericalDistributions}
	Assume $p\geq2$ and $q\geq1$.
	\begin{enumerate}[(1)]
		\item\label{lem:SphericalDistributions1} $\varphi_s$ is entire in $s\in\C$.
		\item\label{lem:SphericalDistributions2} $\varphi_s=\varphi_{-s}$ for all $s\in\C$.
		\item\label{lem:SphericalDistributions3} $\varphi_s=0$ if and only if either $\F=\R$ with $q$ even or $\F=\C,\H,\O$, and $s\in\pm(\rho+2\N)$.
	\end{enumerate}
\end{lem}

In the case where $\varphi_s=0$, i.e. $s=\pm(\rho+2k)$, the space of spherical distributions is two-dimensional and spanned by two distributions $\eta_k$ and $\theta_k$ (see \cite[Chapter VIII~(1)]{faraut1979distributions} and \cite[Chapter 3.11]{Kos83} for their precise definition). As stated in \cite[end of Chapter VIII]{faraut1979distributions} and \cite[Chapter 3.11]{Kos83}:
\begin{equation}
	\left.\frac{d}{ds}\right|_{s=\rho+2k}\varphi_s = \eta_k+\theta_k.\label{eq:DerivativeZetaS}
\end{equation}

\subsection{The inversion formula}

Let
\begin{equation}\label{eq:cfunction}
	c(s)=\frac{\Gamma(dp/2)\Gamma(d
		q/2)}{\sqrt{\pi}}\frac{2^{\rho-s}\Gamma(s)}{\Gamma((s+\rho)/2)\Gamma((s+dp-\rho/2))\Gamma((s+dq-\rho/2))}.
\end{equation}

%

\begin{thm}\label{thm:InversionFormula}
	For $f\in C_c^\infty(X)$ the following inversion formula holds:
	\begin{enumerate}[(1)]
		\item For $\mathbb{F}=\R$ and $q$ odd:
		\begin{multline}
			f(x) = \frac{1}{4\pi} \int_\R (f*\varphi_{i\nu})(x)\frac{d\nu}{\arrowvert c(i\nu)\arrowvert^2}\\
			+\sum_{\rho+2k+1>0}(f*\varphi_{\rho+2k+1})(x)\Res_{s=\rho+2k+1}\frac{1}{c(s)c(-s)}.
		\end{multline}
		\item For $\mathbb{F}=\R$ and $q$ even, or $\mathbb{F}=\C,\H$:
		\begin{multline}
			f(x) = \frac{1}{4\pi}\int_\R (f*\varphi_{i\nu})(x)\frac{d\nu}{\arrowvert c(i\nu)\arrowvert^2}\\
			+\sum_{0<\rho+2k<\rho}(f*\varphi_{\rho+2k})(x)\Res_{s=\rho+2k}\frac{1}{c(s)c(-s)}\\
			+\sum_{\rho+2k\geq \rho}(f*\theta_k)(x)c_{-2}\bigg[\frac{1}{c(s)c(-s)};\rho+2k\bigg].
		\end{multline}
	\end{enumerate}
\end{thm}

Here, $c_{-2}[h(s);s_0]$ denotes the coefficient of $(s-s_0)^{-2}$ in the Laurent expansion of a meromorphic function $h(s)$ around $s_0$.

\begin{proof}
	For $K$-finite $f\in C_c^\infty(X)$ the inversion formula can either be deduced from the Plancherel formula in \cite[Th\'{e}or\`{e}me 10]{faraut1979distributions} and \cite[Theorem 3.13.1]{Kos83}, or can be found in \cite[Theorem 5.2]{Shi95}. Using the estimates in Lemma~\ref{lem:GrowthSphericalDistributions} and \ref{lem:DecayCFunction} it can be extended to all $f\in C_c^\infty(X)$.
\end{proof}

To simplify notation, we put
$$ D=D_1\cup D_2, $$
where
\begin{align*}
	D_1 &= \begin{cases}\{s>0:s\in\rho+2\Z+1\}&\mbox{for $\F=\R$ and $q$ odd,}\\\{s\in(0,\rho):s\in\rho+2\Z\}&\mbox{for $\F=\R$ and $q$ even or $\F=\C,\H,\O$,}\end{cases}\\
	D_2 &= \begin{cases}\emptyset&\mbox{for $\F=\R$ and $q$ odd,}\\\rho+2\N&\mbox{for $\F=\R$ and $q$ even or $\F=\C,\H,\O$,}\end{cases}
\end{align*}
then the inversion formula reads
\begin{multline*}
	f(x) = \frac{1}{4\pi}\int_\R(f*\varphi_{i\nu})(x)\frac{d\nu}{|c(i\nu)|^2} + \sum_{s\in D_1}(f*\varphi_s)(x)\Res_{\sigma=s}\frac{1}{c(\sigma)c(-\sigma)}\\
	+ \sum_{\rho+2k\in D_2}(f*\theta_k)(x)c_{-2}\left[\frac{1}{c(\sigma)c(-\sigma)};\rho+2k\right].
\end{multline*}

\subsection{Images of Poisson transforms}

In order to determine the residue representations in Section~\ref{sec:ResonancesAndRepresentations}, we need to identify the representation of $G$ on the image of the map
$$ C_c^\infty(X)\to C^\infty(X), \quad f\mapsto f*\varphi_s $$
for some values of $s$. We start with the values $s$ for which $s^2-\rho^2$ is an eigenvalue of $-\Box$ on $L^2(X)$.

\begin{prop}[{see \cite[Chapter VIII~(2)]{faraut1979distributions}}]\label{prop:PoissonImages2}
	\begin{enumerate}[(1)]
		\item\label{prop:PoissonImages21} Assume that either $\F=\R$ with $q$ odd and $s>0$, $s\in\rho+2\Z+1$, or that either $\F=\R$ with $q$ even or $\F=\C,\H,\O$ and $0<s<\rho$, $s\in\rho+2\Z$. Then the image of the map $f\mapsto f*\varphi_s$ is the relative discrete series representation $\{u\in L^2(X):\Box u=-(s^2-\rho^2)u\}$.
		\item Assume that $\F=\R$ with $q$ even or $\F=\C,\H,\O$ and $s=\rho+2k\in\rho+2\N$.
		\begin{enumerate}[(a)]
			\item\label{prop:PoissonImages22} The image of the map $f\mapsto f*\eta_k$ is isomorphic to the finite-dimensional representation of $G$ that occurs as subrepresentation of $\calE_{\rho+2k}(\Xi)$ and consists of the $K$-modules $\calY_{\ell,m}$, $\ell+m\leq2k$.
			\item\label{prop:PoissonImages22b} The image of the map $f\mapsto f*\theta_k$ is the relative discrete series representation $\{u\in L^2(X):\Box u=-4k(\rho+k)u\}$ that is equivalent to the unique irreducible quotient of $\calE_{\rho+2k}(\Xi)$ consisting of the $K$-modules $\calY_{\ell,m}$, $\ell-m\geq dq+2k$.
		\end{enumerate}
	\end{enumerate}
\end{prop}

\begin{proof}
	\eqref{prop:PoissonImages21} and \eqref{prop:PoissonImages22b} follow from the inversion formula in Theorem~\ref{thm:InversionFormula}, and \eqref{prop:PoissonImages22} is proven in \cite[Chapter VIII~(1)]{faraut1979distributions} and \cite[proof of Lemma 3.11.2]{Kos83}.
\end{proof}

We also need to study the image of $f\mapsto f*\varphi_s$ for some other values of $s$.

\begin{prop}\label{prop:PoissonImages1}
	\begin{enumerate}[(1)]
		\item\label{prop:PoissonImages1a} Assume that either $\F=\R$ and $p-q\in4\Z+2$ or that $\F=\C$ and $p-q\in2\Z+1$. For $q>1$ and $s=0$, the unitary degenerate principal series $\calE_0(\Xi)$ decomposes into the direct sum $\calE_0^+(\Xi)\oplus\calE_0^-(\Xi)$ of two irreducible subrepresentations consisting of the $K$-modules
		$$ \calE_0^+(\Xi) = \widehat{\bigoplus_{\ell-m+d(p-q)\geq0}}\calY_{\ell,m} \qquad \mbox{and} \qquad \calE_0^-(\Xi) = \widehat{\bigoplus_{\ell-m+d(p-q)\leq-2}}\calY_{\ell,m}. $$
		The image of the map $f\mapsto f*\varphi_s$ is equivalent to $\calE_0^+(\Xi)$. For $q=1$, the unitary degenerate principal series $\calE_0(\Xi)$ is irreducible and the image of the map $f\mapsto f*\varphi_s$ is equivalent to $\calE_0(\Xi)$.
		\item\label{prop:PoissonImages1b} Assume that $\F=\R$ with $p$ odd and $q$ even and $s\in\rho+2\Z+1$, $s>0$. The image of the map $f\mapsto f*\varphi_s$ is equivalent to the subrepresentation of $\calE_s(\Xi)$ with $K$-types $\calY_{\ell,m}$ for $m-\ell<s-\rho+p$ occurring with multiplicity one. This subrepresentation is not unitarizable.
		\item\label{prop:PoissonImages1c} Assume that $\F=\R$ with $p$ even and $q$ odd and $s\in\rho+2\Z$, $s>0$. The image of the map $f\mapsto f*\varphi_s$ is for $0<s<\rho$ equivalent to the subrepresentation of $\calE_s(\Xi)$ with $K$-types $\calY_{\ell,m}$ for $m-\ell<s-\rho+p$ occurring with multiplicity one, and for $s\geq\rho$ equivalent to the finite-dimensional subrepresentation of $\calE_s(\Xi)$ with $K$-types $\calY_{\ell,m}$ for $\ell+m\leq s-\rho$ occurring with multiplicity one. Apart from the trivial representation which occurs for $s=\rho$, these subrepresentations are not unitarizable.
	\end{enumerate}
\end{prop}

\begin{proof}
	The statements about the irreducibility and unitarizability of the subrepresentations with given $K$-types can for instance be deduced from the statements in \cite{HT93}. To show that the image of $f\mapsto f*\varphi_s$ is the claimed representation in all cases, we apply \eqref{eq:DefZeta} together with \eqref{eq:FourierOnKfinite} and \eqref{eq:PoissonOnKfinite}. This reduces the proof to the study of the zeros of $(\ell,m)\mapsto\beta_{\ell,m}(s)$ and $\beta_{\ell,m}(-s)$.
	\begin{enumerate}[(1)]
		\item For $s=0$ the numbers $\beta_{\ell,m}(s)$ vanish if and only if $\ell-m+d(p-q)+2\leq0$, so $\calE_0^+(\Xi)$ is the image of the Fourier transform $\calF_0$ and $\calE_0^-(\Xi)$ is the kernel of the Poisson transform $\calP_0$.
		\item For $s=\rho+2k+1>0$ we have $\beta_{\ell,m}(s)=0$ if and only if $m-\ell\geq p+2k+1$, so the image of the Fourier transform $\calF_s$ consisting of the $K$-modules $\calY_{\ell,m}$ with $m-\ell<p+2k+1$. On the other hand, $\beta_{\ell,m}(-s)\neq0$ for all $\ell,m$ since $s>0$, so $\calP_s$ is an isomorphism.
		\item For $s=\rho+2k$ we have $\beta_{\ell,m}(s)=0$ if and only if either $\ell+m>2k\geq0$ or $m-\ell\geq2k+p$. This shows that the image of the Fourier transform $\calF_s$ consisting of the $K$-modules $\calY_{\ell,m}$ with $\ell+m\leq2k$ in the case $k\geq0$ and $m-\ell<2k+p$ in the case $k<0$. Moreover, $\beta_{\ell,m}(-s)\neq0$ for all $\ell,m$, so $\calP_s$ is an isomorphism.\qedhere
	\end{enumerate}
\end{proof}

\section{Meromorphic continuation}

We now prove the meromorphic continuation of the resolvent operator
$$ \widetilde{R}(\zeta)=(\Box-\rho^2-\zeta^2)^{-1}:C_c^\infty(X)\to\calD'(X). $$
Applying the resolvent $\widetilde{R}(\zeta)$ to the inversion formula in Theorem~\ref{thm:InversionFormula} yields
\begin{equation}
	\widetilde{R}(\zeta)f(x) = I(\zeta)f(x) - D(\zeta)f(x),\label{eq:ResolventFormula1}
\end{equation}
where
\begin{equation}
	I(\zeta)f(x) = \frac{1}{4\pi}\int_\R (\nu^2-\zeta^2)^{-1}(f*\varphi_{i\nu})(x)\frac{d\nu}{|c(i\nu)|^2}\label{eq:ResolventFormula1cont}
\end{equation}
is the contribution of the continuous spectrum and
\begin{multline}
	D(\zeta) = \sum_{s\in D_1}\frac{(f*\varphi_s)(x)}{\zeta^2+s^2}\Res_{\sigma=s}\frac{1}{c(\sigma)c(-\sigma)}\\
	+\sum_{\rho+2k\in D_2}\frac{(f*\theta_k)(x)}{\zeta^2+(\rho+2k)^2}c_{-2}\left[\frac{1}{c(\sigma)c(-\sigma)};\rho+2k\right]\label{eq:ResolventFormula1disc}
\end{multline}
the contribution of the discrete part. While $D(\zeta)$ clearly is meromorphic in $\zeta\in\C$, the integral expression \eqref{eq:ResolventFormula1cont} only shows that $I(\zeta)$ is holomorphic in the upper half plane $\{\IM\zeta>0\}$ and we have to shift the contour to extend it to all $\zeta\in\C$. Contour shifts are possible by Cauchy's Integral Theorem since $(f*\varphi_s)(x)$ is holomorphic in $s\in\C$ by Lemma~\ref{lem:SphericalDistributions}~\eqref{lem:SphericalDistributions1} and $(c(s)c(-s))^{-1}$ is meromorphic as a quotient of gamma functions (see \eqref{eq:cfunction}). Further, we can show that the boundary terms at infinity vanish using the following statements on the growth/decay of $(f*\varphi_s)(x)$ and $(c(s)c(-s))^{-1}$ as $|\IM s|\to\infty$:

\begin{lem}\label{lem:GrowthSphericalDistributions}
	Let $f\in C_c^\infty(X)$. For every compact subset $\Omega\subseteq X$, $R>0$ and $M\in\N$:
	\begin{equation}\label{eq:estimates}
	  \sup_{x\in \Omega,|\RE s|\leq R} (1+|s|)^Me^{-\frac{\pi}{2}|\IM s|}|(f*\varphi_s)(x)| < \infty. 
	\end{equation}
\end{lem}

We give a proof of this statement in Appendix~\ref{app:PaleyWiener}.

\begin{lem}\label{lem:DecayCFunction}
	For every $R>0$ there exist $C>0$ such that
	$$ \frac{1}{|c(s)c(-s)|} \leq C(1+|\IM s|)^\rho e^{-\frac{\pi}{2}|\IM s|} \qquad \mbox{whenever }|\RE s|\leq R. $$
\end{lem}

\begin{proof}
	This follows from Lemma~\ref{lem:Gamma1}.
\end{proof}

By the previous considerations, the Cauchy Integral Formula can be applied to \eqref{eq:ResolventFormula1cont} in order to shift the contour from $\R$ to $\R-iN$, $N>0$. Let $E_N$ denote the set of singularities of $(c(i\nu)c(-i\nu))^{-1}\varphi_{i\nu}$ in the strip $\{-N<\IM\zeta<0\}$. We fix $N>0$ and assume that $\IM\zeta\in(0,N)$ and that $\zeta$ avoids the discrete set of poles of $(c(i\nu)c(-i\nu))^{-1}$. Changing the contour from $\R$ to $\R-iN$, we pick up residues of the function $(c(i\nu)c(-i\nu))^{-1}(f*\varphi_{i\nu})(x)$ at $\nu\in E_N$ as well as a residue at the pole of the function $(\nu^2-\zeta^2)^{-1}=(\nu+\zeta)^{-1}(\nu-\zeta)^{-1}$ which is $\nu=-\zeta$:
\begin{multline}\label{eq:integralres3}
	I(\zeta)f(x) = \frac{1}{4\pi}\int_{\R-iN}\frac{1}{\nu^2-\zeta^2}(f*\varphi_{i\nu})(x)\frac{d\nu}{c(i\nu)c(-i\nu)}\\
	+\frac{i}{4\zeta}\frac{(f*\varphi_{-i\zeta})(x)}{c(i\zeta)c(-i\zeta)}+\frac{1}{2}\sum_{s\in E_N}\Res_{\sigma=s}\frac{1}{\zeta^2+\sigma^2}\frac{(f*\varphi_\sigma)(x)}{c(\sigma)c(-\sigma)}.
\end{multline}

\begin{rmrk}\label{rem:ResidueForSingleDoublePole}
	We will later see that the function $(c(s)c(-s))^{-1}\varphi_s$ can have poles of order $>1$. Here we use the convention $\Res_{\sigma=s}f(\sigma)=c_{-1}(f;s)$ for a meromorphic function $f$ around $\sigma$ with Laurent series expansion $f(\sigma)=\sum_{n=-N}^\infty c_n(f;s)(\sigma-s)^n$. Note that if $(c(\sigma)c(-\sigma))^{-1}\varphi_\sigma$ has a simple pole at $\sigma=s$ then
	$$ \Res_{\sigma=s}\frac{1}{\zeta^2+\sigma^2}\frac{(f*\varphi_\sigma)(x)}{c(\sigma)c(-\sigma)} = \frac{1}{\zeta^2+s^2}\Res_{\sigma=s}\frac{(f*\varphi_\sigma)(x)}{c(\sigma)c(-\sigma)}, $$
	while if $(c(\sigma)c(-\sigma))^{-1}\varphi_\sigma$ has a double pole at $\sigma=s$ then
	\begin{multline*}
		\Res_{\sigma=s}\frac{1}{\zeta^2+\sigma^2}\frac{(f*\varphi_\sigma)(x)}{c(\sigma)c(-\sigma)} = \frac{1}{\zeta^2+s^2}c_{-1}\left[\frac{(f*\varphi_\sigma)(x)}{c(\sigma)c(-\sigma)};\sigma=s\right]\\- \frac{2s}{(\zeta^2+s^2)^2}c_{-2}\left[\frac{(f*\varphi_\sigma)(x)}{c(\sigma)c(-\sigma)};\sigma=s\right].
	\end{multline*}
\end{rmrk}

The expression \eqref{eq:integralres3} is clearly meromorphic in $\{-N<\IM\zeta<N\}$, and since $N>0$ was arbitrary, this shows that $I(\zeta)$ has a meromorphic extension to $\zeta\in\C$. Together with \eqref{eq:ResolventFormula1}, this implies:

\begin{thm}\label{thm:MeroCont}
	The resolvent $\widetilde{R}(\zeta):C_c^\infty(X)\to\calD'(X)$ has a meromorphic extension to $\zeta\in\C$.
\end{thm}

\section{Resonances and residue representations}\label{sec:ResonancesAndRepresentations}

We now study the location of the poles of the meromorphic extension of $\widetilde{R}(\zeta)$ and the corresponding residue representations. In view of \eqref{eq:ResolventFormula1}, \eqref{eq:ResolventFormula1disc} and \eqref{eq:integralres3} the possible poles of $\widetilde{R}(\zeta)$ can be divided into three families which we treat separately.

\subsection{Poles in the upper half plane}\label{sec:PolesUpperHalfPlane}

Since the expression \eqref{eq:ResolventFormula1cont} for $I(\zeta)$ is holomorphic in the upper half plane $\{\IM\zeta>0\}$, the poles of $\widetilde{R}(\zeta)=I(\zeta)-D(\zeta)$ can be read off from the expression \eqref{eq:ResolventFormula1disc} for $D(\zeta)$. They are the \emph{trivial resonances} coming from the discrete spectrum of the Laplacian.

\begin{thm}\label{thm:ResonancesUpperHalfPlane}
	The poles of $\widetilde{R}(\zeta)$ in the upper half plane $\{\IM\zeta>0\}$ are the values $\zeta=is$ with $s\in D$. The corresponding residue representation at $\zeta=is$ is the relative discrete series representation of $G$ on the eigenspace
	$$ \{u\in L^2(X):\Box u=-(s^2-\rho^2)u\}. $$
\end{thm}

\begin{rmrk}
	At first glance, it seems that \eqref{eq:integralres3} suggests that $I(\zeta)$ also has poles in the upper half plane. However, the poles of the second and the third term in \eqref{eq:integralres3} cancel at all values $\zeta$, $\IM\zeta>0$, where $(c(i\zeta)c(-i\zeta))^{-1}(f*\varphi_{-i\zeta})(x)$ has a pole.
\end{rmrk}

\subsection{Poles on the real line}

The only term in \eqref{eq:integralres3} that could produce a pole at $\zeta\in\R$ is the term
$$ \frac{i}{4\zeta}\frac{(f*\varphi_{-i\zeta})(x)}{c(i\zeta)c(-i\zeta)} $$
which might have a pole at $\zeta=0$. Since $\varphi_s$ is non-zero at $s=0$ by Lemma~\ref{lem:SphericalDistributions}, this term in fact has a pole at $s=0$ if and only if $(c(s)c(-s))^{-1}$ vanishes at $s=0$. Using $\Gamma(z)\Gamma(1-z)=\frac{\pi}{\sin\pi z}$, we can rewrite
\begin{equation}
	\frac{1}{c(s)c(-s)} = \const\times\frac{\Gamma(\frac{s+\rho}{2})\Gamma(\frac{-s+\rho}{2})}{\Gamma(s)\Gamma(-s)\sin(\frac{s+dp-\rho}{2}\pi)\sin(\frac{s+dq-\rho}{2}\pi)}\label{eq:PlancherelDensity1}
\end{equation}

\begin{lem}\label{lem:ResolventRewritten}
$(c(s)c(-s))^{-1}$ is regular for all $s\in i\R$, and it vanishes at $s=0$ if and only if either
\begin{itemize}
	\item $\F=\R$ and $p-q\not\equiv2\mod4$, or
	\item $\F=\C$ and $p-q\not\equiv1\mod2$, or
	\item $\F=\H,\O$.
\end{itemize}
\end{lem}

\begin{proof}
We use the expression \eqref{eq:PlancherelDensity1}. First note that the gamma factors $\Gamma(\pm s)$ have a single pole at $s=0$ while the gamma factors $\Gamma(\frac{\pm s+\rho}{2})$ are regular at $s=0$. Further, for $s=0$ the two sine terms become $\sin(\frac{dp-\rho}{2}\pi)$ and $\sin(\frac{dq-\rho}{2}\pi)$, and since $dq-\rho=2-(dp-\rho)$ they are either both non-zero or both zero, the latter being the case iff $dp-\rho=\frac{d(p-q)}{2}+1\in2\Z$. This show the claim.
\end{proof}

\begin{thm}\label{thm:ResonancesRealLine}
	The only possible pole of the resolvent $\widetilde{R}(\zeta)$ for $\zeta\in\R$ is at $\zeta=0$, and this is indeed a pole if and only if either $\F=\C$ and $p-q\in2\Z$ or if $\F=\R$ and $p-q\in4\Z+2$. In this case the corresponding residue representation is equivalent to the representation $\calE_0^+(\Xi)$ for $q>1$, and $\calE_0(\Xi)$ for $q=1$ (see Proposition~\ref{prop:PoissonImages1}~\eqref{prop:PoissonImages1a}).
\end{thm}

\begin{proof}
	The first part of the statement follows from Lemma~\ref{lem:ResolventRewritten}. For the second part, we note that the residue representation is the image of $f\mapsto f*\varphi_0$, so the statement follows from Proposition~\ref{prop:PoissonImages1}~\eqref{prop:PoissonImages1a}.
\end{proof}

\subsection{Poles in the lower half plane}

By \eqref{eq:ResolventFormula1disc} and \eqref{eq:integralres3}, the singularities of $\widetilde{R}(\zeta)=I(\zeta)-D(\zeta)$ in the lower half plane $\{\IM\zeta<0\}$ are given by the following formal expression:
\begin{multline}
	\frac{i}{4\zeta}\frac{(f*\varphi_{-i\zeta})(x)}{c(i\zeta)c(-i\zeta)}+\frac{1}{2}\sum_{s\in E}\Res_{\sigma=s}\frac{1}{\zeta^2+\sigma^2}\frac{(f*\varphi_\sigma)(x)}{c(\sigma)c(-\sigma)}\\
	-\sum_{s\in D_1}\frac{(f*\varphi_s)(x)}{\zeta^2+s^2}\Res_{\sigma=s}\frac{1}{c(\sigma)c(-\sigma)}\\
	-\sum_{\rho+2k\in D_2}\frac{(f*\theta_k)(x)}{\zeta^2+(\rho+2k)^2}c_{-2}\left[\frac{1}{c(\sigma)c(-\sigma)};\rho+2k\right],\label{eq:SingularitiesLowerHalfPlane}
\end{multline}
where $E=\bigcup_{N>0}E_N$ denotes the set of poles of the function $(c(i\nu)c(-i\nu))^{-1}\varphi_{i\nu}$ in the lower half plane. It follows that the singularities of $\widetilde{R}(\zeta)$ in the lower half plane are contained in the set $D\cup E$. To determine $E$ explicitly we need the following result:

\begin{lem}\label{lem:SingularitiesCFunction}
	The singularities of the function $\nu\mapsto(c(i\nu)c(-i\nu))^{-1}$ in the half plane $\{\IM\nu<0\}$ are:
	\begin{itemize}
		\item For $d=1$ with $pq$ even or $d=2,4,8$ there are single poles at $\nu=-iy$, $0<y<\rho$, $y\in\rho+2\Z$, and double poles at $\nu=-iy$, $y\in\rho+2\N$. For $d=1$ with $p-q\in2\Z+1$ there are additional single poles at $\nu=-iy$, $y>0$, $y\in\rho+2\Z+1$.
		\item For $d=1$ with $pq$ odd there are single poles at $\nu=-iy$, $y>0$, $y\in\rho+2\Z+1$.
	\end{itemize}
\end{lem}

\begin{proof}
	Write $\nu=-iy$ with $\RE y>0$. Then \eqref{eq:PlancherelDensity1} becomes
	\begin{equation}
		\frac{1}{c(i\nu)c(-i\nu)} = \const\times\frac{\Gamma(\frac{y+\rho}{2})\Gamma(\frac{-y+\rho}{2})}{\Gamma(y)\Gamma(-y)\sin(\frac{y+dp-\rho}{2}\pi)\sin(\frac{y+dq-\rho}{2}\pi)}.\label{eq:CFunctionLowerHalfPlane}
	\end{equation}
	Note first that $\Gamma(\frac{y+\rho}{2})$ is non-singular for all $y>0$ since $\rho>0$. Further, $\Gamma(\frac{-y+\rho}{2})$ has single poles at $y\in\rho+2\N$. Moreover, $\sin(\frac{y+dp-\rho}{2}\pi)^{-1}$ has a pole if and only if $y+(p-q)\frac{d}{2}+1\in2\Z$, and $\sin(\frac{y+dq-\rho}{2}\pi)^{-1}$ has a pole if and only if $y+(q-p)\frac{d}{2}+1\in2\Z$. This gives the three sets $\rho+2\N$, $(p-q)\frac{d}{2}+1+2\Z$ and $(q-p)\frac{d}{2}+1+2\Z$, each contributing a single pole. On the other hand, the only factor of \eqref{eq:CFunctionLowerHalfPlane} that can produce a zero is $\Gamma(-y)^{-1}$, and it is zero if and only if $y\in\N$. A careful analysis of the cases $d=1,2,4,8$ shows the claim.
\end{proof}

Combining Lemma~\ref{lem:SingularitiesCFunction} with Lemma~\ref{lem:SphericalDistributions} immediately yields:

\begin{coro}\label{cor:SingularitiesZetaOverC}
The singularities of the function $\nu\mapsto(c(i\nu)c(-i\nu))^{-1}\varphi_{i\nu}$ in the half plane $\{\IM\nu<0\}$ are:
\begin{itemize}
	\item For $d=1$ with $q$ even or $d=2,4,8$ there are single poles at $\nu=-iy$, $y>0$, $y\in\rho+2\Z$. For $d=1$ with $p$ odd and $q$ even there are additional single poles at $\nu=-iy$, $y>0$, $y\in\rho+2\Z+1$.
	\item For $d=1$ with $p$ and $q$ odd there are single poles at $\nu=-iy$, $y>0$, $y\in\rho+2\Z+1$.
	\item For $d=1$ with $p$ even and $q$ odd there are single poles at $\nu=-iy$, $0<y<\rho$, $y\in\rho+2\Z$ and at $\nu=-iy$, $y>0$, $y\in\rho+2\Z+1$, and double poles at $\nu=-iy$, $y\in\rho+2\N$.
\end{itemize}
\end{coro}

We now use Remark~\ref{rem:ResidueForSingleDoublePole} to compute the residues of \eqref{eq:SingularitiesLowerHalfPlane}.

\subsubsection{The case $\F=\R$ with $p$ and $q$ even and $\F=\C,\H,\O$}\label{sec:ResonancesCHO}

We first note that by Corollary~\ref{cor:SingularitiesZetaOverC} the first two terms in \eqref{eq:SingularitiesLowerHalfPlane} have poles at the same places $\zeta=-is$, $s\in E$. The corresponding residue is
\begin{multline}
	\frac{i}{4(-is)}\Res_{\zeta=-is}\frac{(f*\varphi_{-i\zeta})(x)}{c(i\zeta)c(-i\zeta)}+\frac{1}{2}\frac{1}{(-is)-is}\Res_{\sigma=s}\frac{(f*\varphi_\sigma)(x)}{c(\sigma)c(-\sigma)}\\
	= \frac{i}{2s}\Res_{\sigma=s}\frac{(f*\varphi_\sigma)(x)}{c(\sigma)c(-\sigma)}.\label{eq:ResidueIZeta12}
\end{multline}
Moreover $E=D=D_1\cup D_2$, so the third and the fourth term in \eqref{eq:SingularitiesLowerHalfPlane} have poles at the same places. We look at $s\in D_1$ and $s\in D_2$ separately:
\begin{itemize}
	\item For $s\in D_1$, the third term in \eqref{eq:SingularitiesLowerHalfPlane} has a single pole at $\zeta=-is$ with residue
	$$ -\frac{(f*\varphi_s)(x)}{(-is)-is}\Res_{\sigma=s}\frac{1}{c(\sigma)c(-\sigma)} = -\frac{i}{2s}(f*\varphi_s)(x)\Res_{\sigma=s}\frac{1}{c(\sigma)c(-\sigma)}. $$
	Since in this case, $\varphi_s\neq0$, this term cancels with \eqref{eq:ResidueIZeta12}, so $\widetilde{R}(\zeta)$ has no pole at $\zeta=-is$ for $s\in D_1$.
	\item For $s=\rho+2k\in D_2$, the fourth term in \eqref{eq:SingularitiesLowerHalfPlane} has a single pole at $\zeta=-is$ with residue
	$$ -\frac{(f*\theta_k)(x)}{(-is)-is}c_{-2}\left[\frac{1}{c(\sigma)c(-\sigma)};\rho+2k\right] = -\frac{i(f*\theta_k)(x)}{2s}c_{-2}\left[\frac{1}{c(\sigma)c(-\sigma)};\rho+2k\right]. $$
	To compare this term with \eqref{eq:ResidueIZeta12}, we rewrite the latter using \eqref{eq:DerivativeZetaS}:
	$$ \Res_{\sigma=s}\frac{(f*\varphi_\sigma)(x)}{c(\sigma)c(-\sigma)} = \left((f*\eta_k)(x)+(f*\theta_k)(x)\right)c_{-2}\left[\frac{1}{c(\sigma)c(-\sigma)};s\right]. $$
	It follows that the terms with $(f*\theta_k)(x)$ cancel and the residue of \eqref{eq:SingularitiesLowerHalfPlane} at $\zeta=-is=-i(\rho+2k)$ equals
	$$ \frac{i}{2s}(f*\eta_k)(x)c_{-2}\left[\frac{1}{c(\sigma)c(-\sigma)};s\right]. $$
\end{itemize}

Together with Proposition~\ref{prop:PoissonImages2}~\eqref{prop:PoissonImages22} this proves:

\begin{thm}\label{thm:ResonancesLowerHalfPlane1}
	For $\F=\R$ with $p$ and $q$ even and $\F=\C,\H,\O$, the poles of $\widetilde{R}(\zeta)$ in the lower half plane $\{\IM\zeta<0\}$ are the values $\zeta\in-i(\rho+2\N)$. The corresponding residue representation at $\zeta=-i(\rho+2k)$ is the irreducible finite-dimensional representation of $G$ that occurs as a subrepresentation of $\calE_{\rho+2k}(\Xi)$.
\end{thm}

\subsubsection{The case $\F=\R$ with $p$ odd and $q$ even}

In view of Corollary~\ref{cor:SingularitiesZetaOverC}, all computations in Section~\ref{sec:ResonancesCHO} remain valid. In addition, there are simple poles of $(c(i\zeta)c(-i\zeta))^{-1}\varphi_{i\zeta}$ at $\zeta=-is$ with $s>0$ and $s\in\rho+2\Z+1$. Here, the same computation as in \eqref{eq:ResidueIZeta12} shows that the corresponding residue of \eqref{eq:SingularitiesLowerHalfPlane} is 
$$ \frac{i}{2s}\Res_{\sigma=s}\frac{(f*\varphi_\sigma)(x)}{c(\sigma)c(-\sigma)} = \frac{i}{2s}(f*\varphi_s)(x)\Res_{\sigma=s}\frac{1}{c(\sigma)c(-\sigma)}, $$
where we have used that $\varphi_s\neq0$ by Lemma~\ref{lem:SphericalDistributions}. Together with Proposition~\ref{prop:PoissonImages1}~\eqref{prop:PoissonImages1b} this shows:

\begin{thm}\label{thm:ResonancesLowerHalfPlane2}
	For $\F=\R$ with $p$ odd and $q$ even, the poles of $\widetilde{R}(\zeta)$ in the lower half plane $\{\IM\zeta<0\}$ are the values $\zeta=-is$, $s>0$, with either $s\in\rho+2\N$ or $s\in\rho+2\Z+1$. The residue representation at $\zeta=-i(\rho+2k)$ is the irreducible finite-dimensional representation of $G$ that occurs as a subrepresentation of $\calE_{\rho+2k}(\Xi)$, and the residue representation at $\zeta=-is$, $s\in\rho+2\Z+1$, is the irreducible subrepresentation of $\calE_s(\Xi)$ with $K$-types $\calY_{\ell,m}$ with $m-\ell<s-\rho+p$, each occurring with multiplicity one.
\end{thm}

\subsubsection{The case $\F=\R$ with $p$ and $q$ odd}\label{sec:ResidueRepsPoddQodd}

By Corollary~\ref{cor:SingularitiesZetaOverC} the singularities of $\zeta\mapsto(c(i\zeta)c(-i\zeta))^{-1}\varphi_{i\zeta}$ in the lower half plane are single poles at $\zeta=-is$, $s>0$ with $s\in\rho+2\Z+1$, so the same computation as in \eqref{eq:ResidueIZeta12} shows that the residue of the first two terms in \eqref{eq:SingularitiesLowerHalfPlane} at $\zeta=-is$ equals
$$ \frac{i}{2s}\Res_{\sigma=s}\frac{(f*\varphi_\sigma)(x)}{c(\sigma)c(-\sigma)}. $$
While the fourth term in \eqref{eq:SingularitiesLowerHalfPlane} vanishes in this case (here $D_2=\emptyset$), the third term also has a single pole at $\zeta=-is$, $s>0$ with $s\in\rho+2\Z+1$, with residue equal to
$$ -\frac{(f*\varphi_s)(x)}{(-is)-is}\Res_{\sigma=s}\frac{1}{c(\sigma)c(-\sigma)} = -\frac{i(f*\varphi_s)(x)}{2s}\Res_{\sigma=s}\frac{1}{c(\sigma)c(-\sigma)}. $$
Since $\varphi_s\neq0$ by Lemma~\ref{lem:SphericalDistributions}, these two residues cancel, so we have:

\begin{thm}\label{thm:ResonancesLowerHalfPlane3}
	For $\F=\R$ with $p$ and $q$ odd, the resolvent $\widetilde{R}(\zeta)$ is holomorphic in the lower half plane $\{\IM\zeta<0\}$.
\end{thm}

\subsubsection{The case $\F=\R$ with $p$ even and $q$ odd}\label{sec:StrangeCase}

As in Section~\ref{sec:ResidueRepsPoddQodd}, we have $D_2=\emptyset$ and hence \eqref{eq:SingularitiesLowerHalfPlane} only contains three terms. In view of Corollary~\ref{cor:SingularitiesZetaOverC}, the potential poles of the expression \eqref{eq:SingularitiesLowerHalfPlane} are at $\zeta=-is$ with either $s\in D_1=(\rho+2\Z+1)\cap(0,\infty)$ or $s\in(\rho+2\Z)\cap(0,\infty)$.

For $s\in\rho+2\Z+1$, $s>0$, the same computations as in Section~\ref{sec:ResidueRepsPoddQodd}, show that all three terms in \eqref{eq:SingularitiesLowerHalfPlane} have a simple pole at $\zeta=-is$, but the residues cancel, so $\widetilde{R}(\zeta)$ is regular at these points.

For $s\in(\rho+2\Z)$, $0<s<\rho$, one can argue as in Section~\ref{sec:ResonancesCHO} that the first two terms in \eqref{eq:SingularitiesLowerHalfPlane} have simple poles at $\zeta=-is$ and the same computation as in \eqref{eq:ResidueIZeta12} produces the following residue of $\widetilde{R}(\zeta)$ at $\zeta=-is$:
$$ \frac{i}{2s}\Res_{\sigma=s}\frac{(f*\varphi_\sigma)(x)}{c(\sigma)c(-\sigma)} = \frac{i(f*\varphi_\sigma)(x)}{2s}\Res_{\sigma=s}\frac{1}{c(\sigma)c(-\sigma)}. $$

Finally, for $s\in\rho+2\N$, the function $(c(\sigma)c(-\sigma))^{-1}$ has a double pole at $\sigma=s$ by Corollary~\ref{cor:SingularitiesZetaOverC}. In view of Remark~\ref{rem:ResidueForSingleDoublePole} we can rewrite the first two terms of \eqref{eq:SingularitiesLowerHalfPlane} as
\begin{multline*}
	\frac{i}{4\zeta}\frac{(f*\varphi_{-i\zeta})(x)}{c(i\zeta)c(-i\zeta)}+\frac{1}{2}\frac{1}{\zeta^2+s^2}c_{-1}\left[\frac{(f*\varphi_\sigma)(x)}{c(\sigma)c(-\sigma)};\sigma=s\right]\\-\frac{1}{2}\frac{2s}{(\zeta^2+s^2)^2}c_{-2}\left[\frac{(f*\varphi_\sigma)(x)}{c(\sigma)c(-\sigma)};\sigma=s\right].
\end{multline*}
The first and the last term in this expression have a pole of order two at $\zeta=-is$, and adding the $c_{-2}$-terms of their Laurent expansion around $\zeta=-is$ gives
$$ c_{-2}\Big[\widetilde{R}(\zeta)f(x);\zeta=-is\Big] = \frac{(f*\varphi_s)(x)}{2s}c_{-2}\left[\frac{1}{c(\sigma)c(-\sigma)};\sigma=s\right]. $$
(We remark at this point that in the same expression the singularities of the three terms around $\zeta=+is$ cancel, so there is no contradiction with the observation from Section~\ref{sec:PolesUpperHalfPlane} that $I(\zeta)$ is holomorphic in the upper half plane.)

Identifying the residue representations using Proposition~\ref{prop:PoissonImages1}~\eqref{prop:PoissonImages1c}, we obtain:

\begin{thm}\label{thm:ResonancesLowerHalfPlane4}
	For $\F=\R$ with $p$ even and $q$ odd, the singularities of $\widetilde{R}(\zeta)$ in the lower half plane $\{\IM\zeta<0\}$ are at the values $\zeta=-is$, $s\in\rho+2\Z$, $s>0$. For $0<s<\rho$ the resolvent $\widetilde{R}(\zeta)$ has a simple pole at $\zeta=-is$ with residue representation isomorphic to the irreducible subrepresentation of $\calE_s(\Xi)$ with $K$-types $\calY_{\ell,m}$ with $m-\ell<s-\rho+p$, each occurring with multiplicity one. For $s\geq\rho$ the resolvent $\widetilde{R}(\zeta)$ has a pole of order two at $\zeta=-is$ with residue representation the irreducible finite-dimensional representation of $G$ that occurs as a subrepresentation of $\calE_s(\Xi)$.
\end{thm}

\appendix

\section{Estimates for the gamma function}

We show some elementary estimates for the classical gamma function $\Gamma(z)$. For fixed $\varepsilon>0$, Stirling's Formula asserts that (see e.g. \cite[Corollary 1.4.3]{AAR99})
\begin{equation}
	\Gamma(z) \sim \sqrt{2\pi}z^{z-\frac{1}{2}}e^{-z} \qquad \mbox{as $|z|\to\infty$ with $|\arg z|\leq\pi-\varepsilon$.}\label{eq:Stirling}
\end{equation}

\begin{lem}[{see \cite[Corollary 1.4.4]{AAR99}}]\label{lem:Gamma1}
	For $K>0$ we have
	$$ |\Gamma(x+iy)| = \sqrt{2\pi}|y|^{x-\frac{1}{2}}e^{-\frac{\pi}{2}|y|}(1+\calO(|y|^{-1})) \qquad \mbox{as $|y|\to\infty$ with $|x|\leq K$.} $$
\end{lem}

\begin{lem}\label{lem:Gamma2}
	For $K>0$ and $a,b\in\R$ there exists $C>0$ such that
	$$ \left|\frac{\Gamma(z+a)}{\Gamma(z+b)}\right| \leq C(1+|z|)^{a-b} \qquad \mbox{as $|z|\to\infty$ with $\RE z\geq -K$.} $$
\end{lem}

\begin{proof}
	Applying Lemma~\ref{lem:Gamma1} for large enough $K$, we may assume that $\RE(z+a),\RE(z+b)>0$. Then $|\arg(z+a)|,|\arg(z+b)|\leq\frac{\pi}{2}$ and hence we can apply \eqref{eq:Stirling} to find
	$$ \frac{\Gamma(z+a)}{\Gamma(z+b)} \sim \frac{e^{(z+a-\frac{1}{2})\ln(z+a)}e^{-z-a}}{e^{(z+b-\frac{1}{2})\ln(z+b)}e^{-z-b}} = e^{b-a}e^{(z+a-\frac{1}{2})\ln(z+a)-(z+b-\frac{1}{2})\ln(z+b)}. $$
	Taking absolute values, writing $z=x+iy$ and using $\ln(w)=\ln|w|+i\arg w$ gives
	\begin{align*}
		\left|\frac{\Gamma(z+a)}{\Gamma(z+b)}\right| &\sim C_1|z+a|^{x+a-\frac{1}{2}}|z+b|^{-x-b+\frac{1}{2}}e^{y(\arg(z+b)-\arg(z+a))}\\
		&\sim C_2(1+|z|)^{a-b}e^{y(\arctan\frac{y}{x+b}-\arctan\frac{y}{x+a})}.
	\end{align*}
	Now,
	\begin{multline*}
		y\left(\arctan\frac{y}{x+b}-\arctan\frac{y}{x+a}\right) = y\arctan\frac{\frac{y}{x+b}-\frac{y}{x+a}}{1+\frac{y^2}{(x+a)(x+b)}}\\
		= y\arctan\frac{(a-b)y}{(x+a)(x+b)+y^2} = \frac{(a-b)y^2}{(x+a)(x+b)+y^2}+\calO\left(|z|^{-2}\right)
	\end{multline*}
	as $|z|\to\infty$ by a Taylor expansion. This expression is obviously bounded as $|z|\to\infty$ and the claim follows.
\end{proof}

\section{Proof of Lemma~\ref{lem:GrowthSphericalDistributions}}\label{app:PaleyWiener}

		Some of the arguments are inspired by the proof of the Paley--Wiener Theorem for the case $\F=\R$ and $p,q\in2\Z$ in \cite{And01}. We first prove estimates for $H\cap K$-invariant functions in a specific $K$-isotypic component and then combine these estimates to show the claim for all compactly supported smooth functions. For $T>0$ we let $C_T^\infty(X)\subseteq C_c^\infty(X)$ denote the space of all $f\in C^\infty(X)$ such that $f(ka_t\cdot x_0)=0$ for $t>T$. Note that $C_T^\infty(X)$ is invariant under the action of $K$.
		\begin{enumerate}[(1)]
			\item Let $C_T^\infty(X)_{\ell,m}$ denote the $K$-isotypic component of $\calY_{\ell,m}$ in $C_T^\infty(X)$ and let $C_T^\infty(X)^{H\cap K}_{\ell,m}$ denote the subspace of functions invariant under the action of $H\cap K=\operatorname{U}(1;\F)\times\operatorname{U}(p-1;\F)\times\operatorname{U}(q;\F)$. The latter space consists of functions $f$ of the form $f(ka_t\cdot x_0)=Y(kM_0)F(t)$ with $Y\in\calY_{\ell,m}^{H\cap K}$ and $F\in C_c^\infty(\R)$ even or odd (depending on the parity of $m$) with $\supp F\subseteq[-T,T]$. Note that $\dim\calY_{\ell,m}^{H\cap K}=1$ for $m=0$ and $\dim\calY_{\ell,m}^{H\cap K}=0$ for $m>0$, so we may assume $m=0$. We first show that for every $R,T>0$ and $n\in\N$ there exist $C,N>0$ such that
			\begin{equation}\label{eq:PWreduction1}
				|\langle\varphi_s,f\rangle| \leq C(1+|\ell|)^{-n}(1+|\IM s|)^Ne^{\frac{\pi}{2}|\IM s|}\|(1-\Delta_B)^Nf\|_\infty
			\end{equation}
			for all $\ell$ and all $f\in C_T^\infty(X)_{\ell,0}^{H\cap K}$. By \eqref{eq:FourierOnKfinite}, \eqref{eq:PoissonOnKfinite} and \eqref{eq:DefZeta} we have
			$$ \langle\varphi_s,f\rangle = \const\cdot\beta_{\ell,0}(s)\beta_{\ell,0}(-s)Y(eM_0)\int_0^\infty\Psi_{\ell,0}(t,s)F(t)A(t)\,dt, $$
			the constant only depending on $X$ and the normalization of the measures. By the proof of \cite[Lemma 2.3]{Koo75} there exists a constant $C_1>0$ such that
			$$ |\Psi_{\ell,0}(t,s)| \leq C_1(1+t)e^{(|\RE s|-\rho)t} \qquad \mbox{for all $t\geq0$ and all $\ell$}. $$
			Further, $A(t)\sim e^{2\rho t}$ as $t\to\infty$, so for $F\in C_c^\infty(\R)$ with $\supp F\subseteq[-T,T]$ we have
			$$ \left|\int_0^\infty \Psi_{\ell,0}(t,s)F(t)A(t)\,dt\right| \leq C_1\|F\|_\infty(1+T)e^{(|\RE s|+\rho)T}. $$
			This implies
			$$ |\langle\varphi_s,f\rangle| \leq C_1|\beta_{\ell,0}(s)\beta_{\ell,0}(-s)|(1+T)e^{(|\RE s|+\rho)T}\underbrace{\|F\|_\infty\|Y\|_\infty}_{=\|f\|_\infty}. $$
			Using Lemma~\ref{lem:Gamma1} and Lemma~\ref{lem:Gamma2} one can further show that for every $R>0$ there exist $C_2>0$ and $N>0$ such that
			$$ |\beta_{\ell,0}(s)\beta_{\ell,0}(-s)| \leq C_2(1+|\ell|)^N(1+|\IM s|)^N e^{\frac{\pi}{2}|\IM s|}\quad \mbox{for $|\RE s|\leq R$ and all $\ell$.} $$
			Together, this shows that for every $T>0$ and $R>0$ there exists $C_3>0$ and $N>0$ such that for $f(ka_t\cdot x_0)=Y(kM_0)F(t)$ with $Y\in\calY_{\ell,0}^{H\cap K}$ and $F\in C_c^\infty(\R)$ even, $\supp F\subseteq[-T,T]$, we have
			$$ |\langle\varphi_s,f\rangle| \leq C_3(1+|\ell|)^N(1+|\IM s|)^Ne^{\frac{\pi}{2}|\IM s|}\|f\|_\infty \qquad \mbox{whenever }|\RE s|\leq R. $$
			Now consider the Casimir element $\Delta_B$ of $K$ acting as a second order differential operator on $X$. We have $\Delta_Bf=-\ell(\ell+dp-2)f$ for $f\in C_T^\infty(X)_{\ell,0}$. Hence, for every $n\in\N$ we obtain
			\begin{align*}
				|\langle\varphi_s,f\rangle| &= (1+\ell(\ell+dp-2))^{-n}|\langle\varphi_s,(1-\Delta_B)^nf\rangle\\
				&\leq C(1+|\ell|)^{N-2n}(1+|\IM s|)^Ne^{\frac{\pi}{2}|\IM s|}\|(1-\Delta_B)^nf\|_\infty,
			\end{align*}
			for some constant $C>0$. This shows \eqref{eq:PWreduction1} after renaming $N$ and $n$.
			\item We now extend \eqref{eq:PWreduction1} to general $f\in C_c^\infty(X)_{\ell,0}$. The orthogonal projection $\calP_{H\cap K}:C_T^\infty(X)_{\ell,0}\to C_T^\infty(X)^{H\cap K}_{\ell,0}$ (with respect to the $L^2$-inner product) is given by
			$$ \calP_{H\cap K}f(x) = \int_{H\cap K}f(kx)\,dk, $$
			where $dk$ denotes the normalized Haar measure of $H\cap K$. It is clear that
			$$ \|\calP_{H\cap K}f\|_\infty \leq \|f\|_\infty. $$
			Then \eqref{eq:PWreduction1} and the $H$-invariance of $\varphi_s$ implies for arbitrary $f\in C_T^\infty(X)_{\ell,0}$:
			\begin{align*}
				|\langle\varphi_s,f\rangle| &= |\langle\varphi_s,\calP_{H\cap K}f\rangle|\\
				&\leq C(1+|\ell|)^{-n}(1+|\IM s|)^Ne^{\frac{\pi}{2}|\IM s|}\|(1-\Delta_B)^N\calP_{H\cap K}f\|_\infty\\
				&= C(1+|\ell|)^{-n}(1+|\IM s|)^Ne^{\frac{\pi}{2}|\IM s|}\|\calP_{H\cap K}(1-\Delta_B)^Nf\|_\infty\\
				&\leq C(1+|\ell|)^{-n}(1+|\IM s|)^Ne^{\frac{\pi}{2}|\IM s|}\|(1-\Delta_B)^Nf\|_\infty,
			\end{align*}
			where we have used that $\Delta_B$ commutes with the action of $K$ and hence with $\calP_{H\cap K}$.
			\item Next, we show that \eqref{eq:PWreduction1} implies the following claim: for all $R,T>0$ there exist $C,N>0$ such that
			\begin{equation}\label{eq:PWreduction2}
				|\langle\varphi_s,f\rangle| \leq C(1+|\IM s|)^Ne^{\frac{\pi}{2}|\IM s|}\|(1-\Delta_B)^Nf\|_\infty \qquad \mbox{for }|\RE s|\leq R,
			\end{equation}
			for all $f\in C_T^\infty(X)$. For $\ell,m$ let $\calP_{\ell,m}:C_T^\infty(X)\to C_T^\infty(X)$ denote the projection onto the isotypic component of $\calY_{\ell,m}$, i.e.
			$$ \calP_{\ell,m}f(x) = d_{\ell,m}\int_K\chi_{\ell,m}(k)^{-1}f(kx)\,dk, $$
			where $\chi_{\ell,m}$ is the character of $\calY_{\ell,m}$, $d_{\ell,m}=\dim\calY_{\ell,m}$ its dimension, and $dk$ the normalized Haar measure on $K$. Note that
			$$ \|\calP_{\ell,m}f\|_\infty \leq d_{\ell,m}\|f\|_\infty \qquad \mbox{for all }f\in C_c^\infty(X). $$
			Write $f\in C_T^\infty(X)$ as $f=\sum_{\ell,m}\calP_{\ell,m}f$ and note that $\langle\varphi_s,P_{\ell,m}\rangle=0$ for $m>0$. Then we can estimate as follows:
			\begin{align*}
				|\langle\varphi_s,f\rangle| &\leq \sum_{\ell}|\langle\varphi_s,\calP_{\ell,0}f\rangle|\\
				&\leq C(1+|\IM s|)^{N}e^{\frac{\pi}{2}|\IM s|}\sum_{\ell}(1+|\ell|)^{-n}\|(1-\Delta_B)^N\calP_{\ell,0}f\|_\infty\\
				&= C(1+|\IM s|)^{N}e^{\frac{\pi}{2}|\IM s|}\sum_{\ell}(1+|\ell|)^{-n}\|\calP_{\ell,0}(1-\Delta_B)^Nf\|_\infty\\
				&\leq C(1+|\IM s|)^{N}e^{\frac{\pi}{2}|\IM s|}\|(1-\Delta_B)^Nf\|_\infty\underbrace{\sum_{\ell}(1+|\ell|)^{-n}d_{\ell,0}}_{<\infty\textup{ for $n$ sufficiently large}},
			\end{align*}
			where the last claim follows from the fact that $d_{\ell,0}$ is polynomial in $\ell$. This implies \eqref{eq:PWreduction2}.
			\item Since $\Box\varphi_s=-(s^2-\rho^2)\varphi_s$ and $|s^2-\rho^2|\sim(1+|\IM s|)^2$ for $|\IM s|$ sufficiently large and $|\RE s|\leq R$, we have for arbitrary $k\in\N$:
			\begin{align*}
				|\langle\varphi_s,f\rangle| &= |s^2-\rho^2|^{-k}|\langle\Box^k\varphi_s,f\rangle|\\
				&= |s^2-\rho^2|^{-k}|\langle\varphi_s,\Box^kf\rangle|\\
				&\leq C(1+|\IM s|)^{N-k}e^{\frac{\pi}{2}|\IM s|}\|(1-\Delta_B)^N\Box^kf\|_\infty,
			\end{align*}
			for a constant $C>0$ depending on $k$. This shows that for all $R,T,M>0$ there exist $C,N>0$ such that
			\begin{equation}\label{eq:PWreduction3}
				|\langle\varphi_s,f\rangle| \leq C(1+|\IM s|)^{-M}e^{\frac{\pi}{2}|\IM s|}\|(1-\Delta_B)^N\Box^Nf\|_\infty,
			\end{equation}
			for all $s\in\C$ with $|\RE s|\leq R$ and all $f\in C_T^\infty(X)$.
			\item We conclude the proof by noting that $(f*\varphi_s)(x)=\langle\varphi_s,L_{g^{-1}}f\rangle$, where $x=g\cdot x_0$. Given a compact subset $\Omega\subseteq X$, we can choose a compact subset $\widetilde{\Omega}\subseteq G$ which projects onto $\Omega\subseteq X=G/H$. For fixed $f\in C_c^\infty(X)$, let $T>0$ be large enough such that $L_{g^{-1}}f\in C_T^\infty(X)$ for all $g\in\widetilde{\Omega}$. Then by \eqref{eq:PWreduction3}, for all $R,M>0$ there exist $C,N>0$ such that
			\begin{align*}
				|(f*\varphi_s)(x)| &= |\langle\varphi_s,L_{g^{-1}}f\rangle|\\
				&\leq C(1+|\IM s|)^{-M}e^{\frac{\pi}{2}|\IM s|}\|(1-\Delta_B)^N\Box^NL_{g^{-1}}f\|_\infty,
			\end{align*}
			for all $x=g\cdot x_0$, $g\in\widetilde{\Omega}$ and all $s\in\C$ with $|\RE s|\leq R$. Since the map $g\mapsto\|(1-\Delta_B)^N\Box^NL_{g^{-1}}f\|_\infty$ is continuous, it is bounded on $\widetilde{\Omega}$ and the proof is complete.\qed
		\end{enumerate}


\providecommand{\bysame}{\leavevmode\hbox to3em{\hrulefill}\thinspace}
\providecommand{\href}[2]{#2}

\contact

\end{document}